\documentclass[11pt]{amsart}
\usepackage[T1]{fontenc}
\usepackage{lmodern}
\usepackage{microtype}
\usepackage{mathtools,amssymb}
\usepackage[margin=1in]{geometry}
\usepackage[hidelinks]{hyperref}

\usepackage[dvipsnames, table]{xcolor}
\usepackage{amsthm}
\theoremstyle{plain}
\usepackage{amssymb}
\usepackage{marvosym}
\usepackage{bm}
\usepackage{mathrsfs}
\usepackage{enumerate}  
\usepackage{mathtools}
\usepackage{tikz-cd}
\usepackage{graphicx}
\usepackage{float}
\usepackage[titletoc,title]{appendix}
\usepackage{marginnote}
\usepackage[disable]{todonotes}
\usepackage{etoolbox}
\usepackage[all]{xy}
\makeatletter
\patchcmd{\Ginclude@eps}{"#1"}{#1}{}{}
\makeatother
\usepackage[outdir=./]{epstopdf}
\usepackage[numbers]{natbib}
\usepackage[utf8]{inputenc}
\usepackage[english]{babel}
\usepackage{changepage}
\usepackage{makecell}
\usepackage{enumitem}

\definecolor{lightblue}{HTML}{1F88CD}
\definecolor{lightgrey}{HTML}{727272}
\definecolor{lightblue2}{HTML}{009EC1}
\definecolor{mypink}{HTML}{FD00B0}
\definecolor{lightred}{HTML}{ff4d4d}

\hypersetup{
    colorlinks=true,
    linkcolor={OliveGreen},
    citecolor={blue},
	urlcolor={black},
    pdftitle={Reconstruction of Fano manifolds from the orthogonal to an exceptional coherent sheaf},
  pdfsubject={A dg-algebra proof of categorical reconstruction}
}

\numberwithin{equation}{section}
\newtheorem{theorem}{Theorem}[section]
\newtheorem{lemma}[theorem]{Lemma}

\newtheorem{corollary}[theorem]{Corollary}
\newtheorem{conjecture}[theorem]{Conjecture}

\theoremstyle{definition}

\theoremstyle{definition}
\newtheorem{remark}[theorem]{Remark}

\newcommand{\cA}{\mathcal A}
\newcommand{\cK}{\mathcal K}
\newcommand{\cQ}{\mathcal Q}

\newcommand{\OO}{\mathcal O}
\newcommand{\C}{\mathbb C}
\newcommand{\Z}{\mathbb Z}
\newcommand{\PP}{\mathbb P}
\newcommand{\D}{\mathrm D}
\newcommand{\Db}{\mathrm D^{\mathrm b}}
\newcommand{\R}{\mathrm R}
\newcommand{\Lder}{\mathrm L}
\newcommand{\Perf}{\operatorname{Perf}}
\newcommand{\RHom}{\operatorname{RHom}}
\newcommand{\Hom}{\operatorname{Hom}}
\newcommand{\HH}{\operatorname{HH}}
\newcommand{\Nat}{\operatorname{Nat}}
\newcommand{\cone}{\operatorname{cone}}
\newcommand{\thick}{\operatorname{thick}}
\newcommand{\id}{\operatorname{id}}
\newcommand{\rk}{\operatorname{rk}}

\title[Reconstruction from exceptional sheaf orthogonals]{Reconstruction of Fano manifolds from the orthogonal to an exceptional coherent sheaf}
\date{October 14, 2026}
\subjclass[2010]{Primary 14F08; Secondary 14J45.}
\keywords{Derived categories, Fano manifolds, exceptional coherent sheaves, reconstruction theorem.}

\author{Songtao Ma}
\address{Department of Mathematics, University of Michigan - Ann Arbor, Ann Arbor, Michigan 48104, United States}
\email{skenma@umich.edu}

\begin{document}

\begin{abstract}
Let $X$ and $Y$ be smooth complex Fano manifolds of dimension at least
three with very ample anticanonical bundles, and let $\mathcal{E}_X$ and $\mathcal{E}_Y$
be exceptional coherent sheaves on $X$ and $Y$, respectively.
We prove that a Fourier--Mukai equivalence
$\mathcal{E}_X^\perp\simeq \mathcal{E}_Y^\perp$ implies $X\simeq Y$.
\end{abstract}

\maketitle

\section{Introduction}

The reconstruction theorem of Bondal and Orlov states that a smooth
projective variety with ample canonical or anticanonical bundle is
determined by its bounded derived category of coherent sheaves
\cite{BondalOrlov}. For a Fano manifold, one may ask whether the entire
derived category is needed. A natural formulation is obtained by
removing an exceptional coherent sheaf. Let $\mathcal{E}_X$ be a coherent sheaf on
$X$ satisfying $\RHom_X(\mathcal{E}_X,\mathcal{E}_X)\simeq\C$. Its right orthogonal
\[
 \mathcal{E}_X^\perp
 =\{F\in\Db(X)\mid\RHom_X(\mathcal{E}_X,F)=0\}
\]
fits into the semiorthogonal decomposition
\[
 \Db(X)=\langle \mathcal{E}_X^\perp,\mathcal{E}_X\rangle.
\]
An equivalence between such orthogonals does not formally give an
equivalence of the ambient derived categories: one must also account for
the gluing to the exceptional object. In the special case
$\mathcal{E}_X=\OO_X$, the following conjecture, attributed to Alexander Perry,
predicts that the orthogonal already determines the variety; see for
example \cite[Conjecture 1.1]{LinRampazzoZhang}.

\begin{conjecture}
Let $X$ and $Y$ be smooth complex Fano manifolds. If there is a
$\C$-linear enhanced equivalence
$\OO_X^\perp\simeq\OO_Y^\perp$, then $X\simeq Y$.
\end{conjecture}

Here an enhanced equivalence means a Morita equivalence of the standard
dg enhancements. In dimension at least three, we prove the following
stronger very ample statement.

\begin{theorem}\label{thm:main}
Let $X$ and $Y$ be smooth complex Fano manifolds of dimension at least
three with very ample anticanonical bundles. Let $\mathcal{E}_X$ and $\mathcal{E}_Y$ be
exceptional coherent sheaves on $X$ and $Y$, respectively. If
$\mathcal{E}_X^\perp \simeq \mathcal{E}_Y^\perp$ by a Fourier-Mukai type equivalence,
then $X\simeq Y$.
\end{theorem}

For structure sheaf orthogonals, the dimension restriction can be removed.

\begin{corollary}\label{cor:structure-sheaf}
Let $X$ and $Y$ be smooth complex Fano manifolds with very ample
anticanonical bundles. If $\OO_X^\perp\simeq\OO_Y^\perp$ by a Fourier-Mukai type equivalence, then $X\simeq Y$.
\end{corollary}

\begin{remark}
    Actually, one can drop the assumption that the equivalence is of Fourier-Mukai type by using the technique developed in \cite{LiPertusiZhaoEnh}. See also \cite[Remark 1.7]{LinZhangFano}.
\end{remark}

Unless stated otherwise, tensor products, Hom complexes, pullbacks, and pushforwards are derived. We use the standard dg enhancements of derived categories and their admissible subcategories, and work with enhanced lifts of Fourier–Mukai equivalences. All dg algebras are unital, and their morphisms preserve units.

\subsection{Related Work} For structure sheaf orthogonals, several cases are covered by earlier
categorical Torelli results.
\cite{Elagin} proves reconstruction for quartic del Pezzo surfaces. More recently, \cite{LinRampazzoZhang} treats del Pezzo
surfaces of degrees one through four , while
\cite{LinZengZhang} proves reconstruction results for index-one
hypersurfaces and certain double covers.
Our argument in dimension at least three does not require a hypersurface presentation or a classification of Fano manifolds.

Categorical Torelli theorems have been established for several
classes of Fano threefolds and fourfolds.
\cite{BMMS} proved that a smooth cubic
threefold is determined by its Kuznetsov component.
\cite{FeyzbakhshLiuZhang} gave a unified treatment for del Pezzo threefolds of Picard rank one and degrees $2$, $3$, and $4$. 

For Gushel--Mukai threefolds, \cite{JacovskisLinLiuZhang}
birational and refined categorical Torelli results.
More generally, \cite{JacovskisLiuZhang} proved a refined categorical Torelli theorem for index-one prime Fano threefolds of genus $g\ge6$. For cubic fourfolds, \cite{HuybrechtsRennemo} proved reconstruction from the Kuznetsov component equipped with its degree-shift functor. Related reconstruction results for Fano hypersurfaces were
obtained in \cite{Pirozhkov} and \cite{LinZhang}.
In particular, the index-one hypersurface case includes smooth quartic threefolds and quintic fourfolds.

After publishing the first version of this paper, the author was informed that Lin and Zhang had also obtained the same result using a similar method, see \cite[Theorem 1.5]{LinZhangFano}.

\subsection{Acknowledgment} The Author thanks Chunyi Li, Zhiyu Liu and Alexander Perry for useful discussion and advices. In particular, the original version of this paper merely proved Corolloary \ref{cor:structure-sheaf}. Due to Chunyi Li's suggestion, the author found that the result can be easily generalized to arbitrary exceptional coherent sheaves. Moreover, conjectures in chapter \ref{sec:further-directions} are also suggested by Chunyi Li.

\subsection{AI Disclosure} The paper is a result of many rounds of discussions with ChatGPT 6 Astra. The majority of the paper is written by the author, while AI tools are involved in proof-reading and revision of the draft. The author is fully responsible for the correctness and accuracy of all mathematical content. 

\section{Anticanonical sections}

From now on, we denote $\cA_X := \mathcal{E}_X^\perp$. Assume that $n=\dim X\ge2$, and write $\mathcal{E}=\mathcal{E}_X$ and
$L=\omega_X^{-1}$. Also, let
$\iota\colon\cA_X\hookrightarrow\Perf(X)$ and $q=\iota^*$.
Choose a classical generator $E$ of $\cA_X$, and set
\begin{equation}\label{eq:notation}
 A=\RHom_X(E,E),\qquad P=A^\vee[-n],\qquad
 C_m=\RHom_{A^e}(P^{\otimes_A m},A)\quad(m\ge1),
\end{equation}
where $A^e=A\otimes A^{\mathrm{op}}$ and $A^\vee=\RHom_\C(A,\C)$.
The algebra $A$ is smooth and proper. Its dual bimodule represents the
Serre functor $S_{\cA_X}$, so $P$ is invertible in $\D(A^e)$.

Since $X$ is smooth, $\mathcal{E}$ is perfect; write $\mathcal{E}^\vee=\R\mathcal Hom_X(\mathcal{E},\OO_X)$.
Let $T_L=(-)\otimes L$. The projection $Q=\iota q$ is left mutation
through $\mathcal{E}$; its kernel is
\[
 \cQ=\cone(\mathcal{E}^\vee\boxtimes \mathcal{E}\longrightarrow\Delta_*\OO_X)
\]
where $\Delta\colon X\hookrightarrow X\times X$ is the diagonal.
Denote by $\cK_m$ the kernel of $(QT_L)^mQ$.

\begin{lemma}\label{lem:kernel}
The natural morphism
\[
 \alpha_m: \RHom_{X\times X}(\Delta_*\OO_X,\Delta_*L^m)
 \longrightarrow C_m
\]
has cone $R_m: = \text{Cone}(\alpha_m)$ in $\D^{\ge n-m-1}(\C)$.
\end{lemma}

\begin{proof}
The functor represented by $P$ is
$p=\iota^!T_{\omega_X}\iota=S_{\cA_X}[-n]$.
Its left adjoint, hence its inverse, is $r=qT_L\iota$.
The extension $F\mapsto\iota Fq$ is fully faithful on enhanced
functor categories: restriction along $\iota$ is inverse on morphism
complexes, by $q\iota\simeq\id$ and the adjunction $q\dashv\iota$.
It preserves composition and takes $\id_{\cA_X}$ to $Q$.
Derived Morita theory \cite[Theorem 8.15]{Toen} gives
\begin{equation}\label{eq:kernels}
 \begin{split}
 C_m
 &\simeq\RHom_{A^e}(A,P^{-m})
 \simeq\RHom_{X\times X}(\cQ,\cK_m)\\
 &\simeq\RHom_{X\times X}(\Delta_*\OO_X,\cK_m).
 \end{split}
\end{equation}
The last isomorphism follows from
$\Phi_{\cK_m}(\mathcal{E})=0$ and the fact that
$\RHom_{X\times X}(\mathcal{E}^\vee\boxtimes \mathcal{E},\cK_m)
 \simeq\RHom_X(\mathcal{E},\Phi_{\cK_m}(\mathcal{E}))$.

The triangle
$\Delta_*\OO_X\to\cQ\to \mathcal{E}^\vee\boxtimes \mathcal{E}[1]$
induces a convolution filtration of $\cK_m$ whose first term is
$\Delta_*L^m$. Its remaining graded pieces are
\begin{equation}\label{eq:pieces}
 \begin{gathered}
 \left(\bigotimes_{\nu=1}^{j-1}
       \RHom_X(\mathcal{E},\mathcal{E}\otimes L^{t_{\nu+1}-t_\nu})\right)
 \otimes\bigl((\mathcal{E}^\vee\otimes L^{t_1})
       \boxtimes(\mathcal{E}\otimes L^{m-t_j})\bigr)[j],\\
 0\le t_1<\cdots<t_j\le m,
 \end{gathered}
\end{equation}
for $1\le j\le m+1$.
Indeed, choosing the $j$ external-product factors divides the $m$
tensor operations into intervals. The interior intervals contribute
$\RHom_X(\mathcal{E},\mathcal{E}\otimes L^{t_{\nu+1}-t_\nu})$.
These complexes have no negative cohomology, since $\mathcal{E}$ and
$\mathcal{E}\otimes L^a$ are coherent sheaves concentrated in degree zero.

For $a,b\ge0$, diagonal duality gives
\[
 \Delta^!\bigl((\mathcal{E}^\vee\otimes L^a)\boxtimes(\mathcal{E}\otimes L^b)\bigr)
 \simeq \mathcal{E}^\vee\otimes \mathcal{E}\otimes L^{a+b+1}[-n].
\]
Consequently,
\[
 \begin{aligned}
 &\RHom_{X\times X}\bigl(\Delta_*\OO_X,
       (\mathcal{E}^\vee\otimes L^a)\boxtimes(\mathcal{E}\otimes L^b)\bigr)[j]\\
 &\hspace{3em}\simeq\RHom_X(\mathcal{E},\mathcal{E}\otimes L^{a+b+1})[j-n].
 \end{aligned}
\]
The right-hand side lies in $\D^{\ge n-j}(\C)$.
Applying $\RHom_{X\times X}(\Delta_*\OO_X,-)$ to the filtration proves
the assertion, since $j\le m+1$.
\end{proof}

For $s\in H^0(X,L)$, let $\theta_s\colon P\to A$ be the bimodule
class induced by multiplication
\[
 \RHom_X(E,E\otimes\omega_X)\xrightarrow{\ s\ }\RHom_X(E,E),
\]
using Serre duality on the source, then we have the following properties for $\theta_s$.

\begin{corollary}\label{cor:sections}
For $k<\min\{0,n-m-1\}$, one has $H^k(C_m)=0$.
The section map
\begin{equation}\label{eq:sections}
 H^0(X,L)\longrightarrow H^0(C_1),\qquad s\longmapsto\theta_s,
\end{equation}
is injective for $n\ge2$ and an isomorphism for $n\ge3$.
\end{corollary}

\begin{proof}
The complex
$\RHom_{X\times X}(\Delta_*\OO_X,\Delta_*L^m)$
has no negative cohomology, and its zeroth cohomology is
$H^0(X,L^m)$. From the triangle 
\[\RHom_{X\times X}(\Delta_*\OO_X,\Delta_*L^m) \xrightarrow{\alpha_m} C_m \xrightarrow{} R_m \rightarrow \RHom_{X\times X}(\Delta_*\OO_X,\Delta_*L^m)[1],\]
Lemma \ref{lem:kernel} gives the vanishing of $H^k(C_m)$ and the
claimed injectivity or bijectivity for the kernel-defined section map $H^0(\alpha_1)$, and the remaining task is to identify it with \eqref{eq:sections}.

Note that the kernel morphism
$\Delta_*\OO_X\xrightarrow{s}\Delta_*L\to\cK_1$
restricts to $a_s := q(\iota\xrightarrow{s}T_L\iota)\colon\id\to r$.
The counit $\varepsilon: rp\to\id$ is induced by
$\iota\iota^!\to\id_{\Perf(X)}$ and the corresponding transformation $p\to\operatorname{id}_{\mathcal A_X}$ is given, for each
$F\in\mathcal A_X$, by
\[
pF\xrightarrow{\,a_{s,pF}\,}rpF
\xrightarrow{\,\varepsilon_F\,}F.
\]
Due to naturality of multiplication by $s$ and of the projection
counit $\iota\iota^!\to\operatorname{id}_{\operatorname{Perf}(X)}$,
this composite equals
\[
\iota^!\bigl(
\iota F\otimes\omega_X\xrightarrow{s}\iota F
\bigr):
pF\longrightarrow F,
\]
using $\iota^!\iota F\simeq F.$ On the generator, this is precisely the multiplication class $\theta_s$.
\end{proof}

\begin{lemma}\label{lem:dimension}
If $n\ge2$ and $H^0(X,L)\ne0$, then
\begin{equation}\label{eq:dimension}
 \min\{d\in\Z\mid
       \Hom_{\D(A^e)}(A^\vee,A[d])\ne0\}=n.
\end{equation}
\end{lemma}

\begin{proof}
Corollary \ref{cor:sections} gives $C_1\in\D^{\ge0}(\C)$ and
$H^0(C_1)\ne0$. Since $Hom_{\D(A^e)}(A^\vee,A[d])=H^{d-n}(C_1),$ the minimum is $n$.
\end{proof}

The groups in \eqref{eq:dimension} are $H^0\Nat(S_{\cA_X},\id[d])$,
so the minimum is intrinsic to the enhanced category.
For the structure-sheaf orthogonal on $X=\PP^1$, the category
$\OO_X^\perp\simeq\Perf(\C)$ gives the value $0$, whereas the orthogonal
for a point is zero. Thus the equivalences in Theorem \ref{thm:main}
and Corollary \ref{cor:structure-sheaf} preserve dimension.

\section{Anticanonical divisors}

We first give the dg comparison used below.
Algebras under $A$ are called quasi-isomorphic under $A$ if they are
connected by a zigzag of dg-algebra quasi-isomorphisms commuting with
the maps from $A$.

\begin{lemma}\label{lem:comparison}
Let $A$ be a dg algebra and $P\in\D(A^e)$ invertible. Assume
\begin{equation}\label{eq:vanishing}
 H^{-1}(C_1)=0,\qquad
 H^{1-2m}(C_m)=H^{2-2m}(C_m)=0\quad(m\ge2).
\end{equation}
Suppose dg-algebra maps $A\to B_i$ fit into triangles
\[
 P\xrightarrow{\theta_i}A\longrightarrow B_i\longrightarrow P[1]
 \quad\text{in }\D(A^e),\qquad i=1,2.
\]
If $\theta_1=\theta_2$, then $B_1$ and $B_2$ are quasi-isomorphic under
$A$, compatibly with the displayed quotient identifications.
\end{lemma}

\begin{proof}
\emph{Step 1: Comparison of the bimodule extensions.}
After taking relative semi-free resolutions, we may assume that
$A\subset B_i$ splits as a graded $A$-bimodule and that
$Q_i=B_i/A$, together with its tensor powers, is cofibrant.
Write $B_i=A\oplus Q_i$, with differential
\[
d_{B_i}=
\begin{pmatrix}
d_A & t_i\\
0 & d_{Q_i}
\end{pmatrix},
\qquad t_i\in\operatorname{Hom}^1_{A^e}(Q_i,A),
\]
where $t_i$ belongs to the degree-1 component of the Hom complex. Choose a quasi-isomorphism $g:Q_1\to Q_2$ representing the prescribed identification of the quotients. Since $\theta_1=\theta_2$, the corresponding extension classes agree, so
\[
t_1-t_2g=d_Ah-hd_{Q_1}
\]
for some $h\in\operatorname{Hom}^0_{A^e}(Q_1,A)$.
Consequently,
\[
f:B_1\longrightarrow B_2,\qquad
f(a,q)=(a+h(q),g(q))
\]
is a dg-bimodule map restricting to $\mathrm{id}_A$ and
inducing $g$ on the quotients. Thus $f$ is a
quasi-isomorphism. It remains to lift it
to dg algebras.

\smallskip
\noindent
\emph{Step2: A relative algebra resolution.}
Let $R$ be the normalized relative bar--cobar algebra of $B_1$ under $A$.
It is generated over $A$ by tensor words $[b_1|\cdots|b_r]$, with tensors
over $A$, of degree $\sum_j |b_j|+1-r$, subject to
\[
    [a]=a\quad(a\in A),\qquad
    [b_1|\cdots|b_r]=0
    \quad\text{if }r>1\text{ and some }b_j\in A.
\]
The differential on $R$ is the standard bar--cobar differential. The normalization relations are preserved
by this differential, and the dg-algebra identities give $d_R^2=0$.
The construction is relative to $A$ and does not require
an augmentation $B_1\to A$. As a graded algebra,
\[
    R=T_A\!\left(\bigoplus_{r\geq1}
        Q_1^{\otimes_A r}[r-1]\right).
\]
On a generator of length $r$, the differential is the sum of the
internal differential on $Q_1^{\otimes_A r}[r-1]$ and terms involving
only elements of $A$ and generators of smaller length.

There is a dg-algebra map
\[
    \epsilon\colon R\longrightarrow B_1,
    \qquad [b]\longmapsto b,
    \qquad [b_1|\cdots|b_r]\longmapsto0\quad(r>1).
\]
Equip $R$ with the increasing filtration $F_\bullet R$
by total tensor length, assigning weight $r$ to
$Q^{\otimes_A r}[r-1]$ and extending additively
to products. Then $F_1R$ is isomorphic to $B_1$
as a complex of $A$-bimodules via $b\mapsto[b]$. For $N\geq2$, the
$N$th associated graded complex is
\[
    Q_1^{\otimes_A N}\otimes_{\mathbb C}
    \left(\mathbb C\xrightarrow{\mathrm{id}}\mathbb C\right)
    ^{\otimes(N-1)},
\]
where each two-term factor lies in degrees $-1$ and $0$ and records
whether a cut is inserted at the corresponding position. These
complexes are contractible. Since $\operatorname{gr}^F_N R$ is acyclic for every $N\ge 2$,
the inclusions $F_1R\hookrightarrow F_NR$ are
quasi-isomorphisms for all $N\ge 1$.
As $R=\varinjlim_N F_NR$ and filtered colimits are exact,
the inclusion $j:B_1\cong F_1R\hookrightarrow R$
is also a quasi-isomorphism. Since $\epsilon$ restricts to the
identity on this subcomplex, it is a quasi-isomorphism.

\smallskip
\noindent
\emph{Step3: Lifting to an algebra map.}
For $r\geq2$, apply
$\mathbf R\operatorname{Hom}_{A^e}(Q_1^{\otimes_A r},-)$
to the triangle $A\to B_2\to P[1]$. Since $Q_1\simeq P[1]$ and $P$
is invertible, the resulting long exact sequence contains
\[
    H^{2-2r}(C_r)\longrightarrow
    H^{2-r}\mathbf R\operatorname{Hom}_{A^e}
        (Q_1^{\otimes_A r},B_2)
    \longrightarrow H^{3-2r}(C_{r-1}).
\]
Both outer terms vanish by hypothesis; when $r=2$, the right-hand term
is $H^{-1}(C_1)$. Consequently,
\[
    H^{2-r}\operatorname{Hom}_{A^e}
        (Q_1^{\otimes_A r},B_2)=0
    \qquad(r\geq2).
\]

We construct $F\colon R\to B_2$ by induction on generator length,
starting with the identity on $A$ and $f$ on length-one generators.
Suppose that $F$ has been defined on generators of length less than
$r$ and commutes with their differentials. Apply $F$ to the
shorter-generator terms in the differential of a length-$r$ generator.
This gives a degree-$(2-r)$ cochain in
$\operatorname{Hom}_{A^e}(Q_1^{\otimes_A r},B_2)$.
It is closed because $d_R^2=0$ and $F$ already commutes with the
differential at shorter lengths. The vanishing above provides a
primitive of degree $1-r$. Taking this primitive as the value of $F$
on $Q_1^{\otimes_A r}[r-1]$ extends $F$ to length $r$ and makes it
commute with the differential there. Induction therefore defines a
dg-algebra map $F$ on all of $R$.

The restriction of $F$ to the subcomplex $B_1\subset R$ is $f$.
Since this inclusion and $f$ are quasi-isomorphisms, so is $F$.
We therefore obtain a zigzag of dg-algebra quasi-isomorphisms under $A$,
\[
    B_1\xleftarrow{\ \epsilon\ }R\xrightarrow{\ F\ }B_2.
\]
Its induced bimodule isomorphism is represented by $f$, so it is
compatible with the prescribed quotient identifications. This completes the proof.
\end{proof}

Return to the geometric notation \eqref{eq:notation}.
For $0\ne s\in H^0(X,L)$, let
$j_s$ be the inclusion of the zero divisor $D_s=V(s)\hookrightarrow X$. Then set $E_s=\Lder j_s^*E$, and
\[
 B_s=\RHom_{D_s}(E_s,E_s).
\]
Restriction of endomorphisms gives a unital dg-algebra morphism $A\to B_s$. No smoothness or reducedness assumption is imposed on $D_s$.

The following calculation is related to the projector-kernel computations for hypersurfaces in \cite[Section 3]{Pirozhkov}. 

\begin{lemma}\label{lem:geometric}
The restriction map $A\to B_s$ fits into a bimodule triangle
\begin{equation}\label{eq:geometric-triangle}
 P\xrightarrow{\theta_s}A\longrightarrow B_s\longrightarrow P[1].
\end{equation}
Moreover, if $\operatorname{Supp}(\mathcal{E})\not\subseteq D_s$, then
$\Perf(B_s)\simeq\Perf(D_s)$.
\end{lemma}

\begin{proof}
Since $X$ is integral and $s\ne0$, the Koszul dg algebra
$K_s=[\omega_X\xrightarrow{s}\OO_X]$, in degrees $-1,0$, resolves
$j_{s*}\OO_{D_s}$. Notice that
\[
 B_s\simeq\R\Gamma\!\left(
       X,\R\mathcal Hom_X(E,E)\otimes K_s\right).
\]
Now consider the bimodule triangle 
\[
R\operatorname{Hom}_X(E,E\otimes\omega_X)
\longrightarrow
R\operatorname{Hom}_X(E,E)
\longrightarrow
B_s
\longrightarrow
R\operatorname{Hom}_X(E,E\otimes\omega_X)[1].
\]
Its first term can be identified with
$A^\vee[-n]=P$ as a bimodule, and its map to
$R\operatorname{Hom}_X(E,E)$ is $\theta_s$ by definition.
This proves \eqref{eq:geometric-triangle}.

Assume $\operatorname{Supp}(\mathcal{E})\not\subseteq D_s$.
Choose $x\in\operatorname{Supp}(\mathcal{E})\setminus D_s$, and put
$M_x=Q(\OO_x)\in\cA_X=\thick(E)$.
Since $\mathcal{E}$ is a nonzero coherent sheaf at $x$, the complex
$W_x:=\RHom_X(\mathcal{E},\OO_x)$ is nonzero. The mutation triangle is
\[
 W_x\otimes \mathcal{E}\longrightarrow\OO_x\longrightarrow M_x.
\]
Restricting to $D_s$ gives
\[
 \Lder j_s^*M_x\simeq
 W_x\otimes\Lder j_s^*\mathcal{E}[1]\in\thick(E_s).
\]
As $W_x$ is a nonzero bounded complex of finite-dimensional vector
spaces, it is a direct sum of shifts of $\C$. Hence
$\Lder j_s^*\mathcal{E}\in\thick(E_s)$.
Since $\Perf(X)=\thick(E,\mathcal{E})$, the restrictions of all perfect
complexes on $X$ lie in $\thick(E_s)$.
In particular, it contains every power of the restriction of a very
ample line bundle. These generate $\Perf(D_s)$ due to \cite[Theorem 4]{OrlovGenerators}.
Thus $\thick(E_s) = \Perf(D_s)$, which gives the stated Morita equivalence.
\end{proof}

For $n\ge3$, the pair $(A,P)$ satisfies the vanishings
\eqref{eq:vanishing}.
Indeed, Corollary \ref{cor:sections} gives $H^{-1}(C_1)=0$.
For $m\ge2$, both $1-2m$ and $2-2m$ are negative, and
\[
 (n-m-1)-(2-2m)=n+m-3>0.
\]
The remaining vanishings follow from the same way.

\section{Proof of the main result}

Now we are ready to prove the main result of the paper. 

\begin{proof}[Proof of Theorem \ref{thm:main}]
Let $\Phi\colon\cA_X\simeq\cA_Y$ be the enhanced equivalence.
By Lemma \ref{lem:dimension}, $\dim X=\dim Y=n$.
Choose $E_X$ and $E_Y=\Phi(E_X)$ as generators.
Their endomorphism algebras are quasi-isomorphic and we can choose a common
cofibrant dg algebra $A$ mapping quasi-isomorphically to both:
\[
A_X \xleftarrow{\sim} A \xrightarrow{\sim} A_Y.
\]
Via the maps $A\to A_X$ and $A\to A_Y$, regard $B_s^X$ and
$B_t^Y$ as dg algebras under $A$.

Now identify the first terms of their triangles \eqref{eq:geometric-triangle} with the bimodule $P=A^\vee[-n]$.
In this way, Corollary \ref{cor:sections} defines a linear isomorphism
\[
 \varphi\colon H^0(X,\omega_X^{-1})
       \xrightarrow{\ \sim\ }H^0(Y,\omega_Y^{-1}).
\]
The boundary classes corresponding to $s$ and $\varphi(s)$ agree.
Lemma \ref{lem:comparison} therefore shows that $B_s^X$ and
$B_{\varphi(s)}^Y$ are quasi-isomorphic under $A$.
For $Z=X,Y$, let
\[
 U_Z:=\{[s]\in\PP H^0(Z,\omega_Z^{-1})\mid
 \operatorname{Supp}(\mathcal{E}_Z)\not\subseteq D_s\}.
\]
Each $U_Z$ is the complement of a proper projective linear subspace.
On the common open set
$U=U_X\cap\PP(\varphi)^{-1}(U_Y)$, Lemma \ref{lem:geometric} gives
\[
 \Perf(D_s)\simeq\Perf(D_{\varphi(s)})\qquad([s]\in U).
\]
A proper scheme over $\C$ is smooth if and only if its perfect dg category is homologically smooth \cite[Proposition 3.31]{OrlovSmooth}. Thus $D_s$ is smooth if and only if $D_{\varphi(s)}$ is smooth for $[s]\in U$. For $Z=X,Y$, set
\[\Sigma_Z:=\{[s]\in\PP H^0(Z,\omega_Z^{-1})\mid D_s\text{ is singular}\}.\]
Hence $\PP(\varphi)$ identifies $\Sigma_X\cap U$ with
$\Sigma_Y\cap\PP(\varphi)(U)$.
For the anticanonical embeddings, $\Sigma_X=X^\vee$ and
$\Sigma_Y=Y^\vee$: the divisor $D_s$ is singular at $x$ precisely
when the corresponding hyperplane contains $T_xX$.
These dual varieties are irreducible and nondegenerate, since the
anticanonical embeddings are smooth, nondegenerate, and nonlinear.
Thus both discriminants meet the common open set densely, and taking
closures gives $\PP(\varphi)(\Sigma_X)=\Sigma_Y$.
Finally, by projective biduality one has $\PP(\varphi^\vee)(Y)=X$.
\end{proof}

\begin{proof}[Proof of Corollary \ref{cor:structure-sheaf}]
Take $\mathcal{E}_X=\OO_X$ and $\mathcal{E}_Y=\OO_Y$, so that
$\cA_X=\OO_X^\perp$ and $\cA_Y=\OO_Y^\perp$.
The equivalence preserves dimension by Lemma \ref{lem:dimension} and
the paragraph following it. In dimension at least three, apply
Theorem \ref{thm:main}.
In dimensions zero and one the only Fano manifolds are a point and $\PP^1$, respectively.

Suppose $\dim X=2$. Then $X$ and $Y$ are del Pezzo surfaces of degree
at least three. Additivity of Hochschild homology
\cite[Theorem 7.3]{Kuznetsov} gives
\[
 \dim_\C\HH_0(\cA_X)=11-K_X^2,
\]
so $K_X^2=K_Y^2=d$.
For $d=3$ and $d=4$, apply
\cite[Corollary 1.6]{LinZhang} and
\cite[Theorem 3.1]{Elagin}, respectively.
For $d=5,6,7,9$, there is a unique complex isomorphism class
\cite[Chapter 8]{Dolgachev}.

In degree eight it remains to distinguish $\PP^1\times\PP^1$ from
$\mathbb F_1$. Write $\operatorname{ch}(F)=(\rk(F),c_1(F),c_2(F))$.
For $F\in\cA_X$, Riemann--Roch and $\chi(F)=0$ give
$c_2(F)=\tfrac12K_X\cdot c_1(F)-\rk(F)$.
Substituting this into the Euler pairing yields, for $F,G\in\cA_X$,
\[ \chi(F,G)+\chi(G,F)
 =-2\rk(F)\rk(G)+\rk(F)K_X\cdot c_1(G)+\rk(G)K_X\cdot c_1(F)-2c_1(F)\cdot c_1(G).\]
On $\PP^1\times\PP^1$, the canonical class is divisible by two,
so the above equation is always even.
On $\mathbb F_1$, let $C$ be the exceptional curve and choose $x\notin C$.
Both $I_x$ and $\OO_C(-1)$ are acyclic, and their symmetric Euler
pairing is $K_X\cdot C=-1$.
An equivalence preserves this pairing, so these two orthogonals are
not equivalent.
\end{proof}

The dimension restriction in Theorem \ref{thm:main} cannot be removed for arbitrary exceptional coherent sheaves, as the following example shows.
\begin{remark}
Let $X$ and $Y$ be nonisomorphic smooth complex del Pezzo
surfaces of degree $4$. Their anticanonical bundles are very ample.
Contracting a $(-1)$-curve on each surface gives morphisms
\[
\pi_X \colon X \longrightarrow S,
\qquad
\pi_Y \colon Y \longrightarrow S
\]
to the same del Pezzo surface $S$ of degree $5$, since such a surface
is unique up to isomorphism over $\mathbb C$
\cite[Chapter~8]{Dolgachev}.

Let $C_X$ and $C_Y$ be the exceptional curves, and set
$E_X=\mathcal O_{C_X}$ and $E_Y=\mathcal O_{C_Y}$, viewed as sheaves on
$X$ and $Y$, respectively. These objects are exceptional, and
Orlov's blow-up formula \cite{OrlovBlowup} gives
\[
D^b(Z)
=
\bigl\langle L\pi_Z^*D^b(S),E_Z\bigr\rangle,
\qquad Z\in\{X,Y\}.
\]
Consequently, $E_X^\perp \simeq D^b(S) \simeq E_Y^\perp.$
The resulting equivalence is the restriction of the Fourier--Mukai
functor $L\pi_Y^*\circ R\pi_{X*}$, although $X\not\simeq Y$.
\end{remark}

\section{Further directions}\label{sec:further-directions}

The first question is whether the very-ampleness assumption in
Theorem~\ref{thm:main} can be removed, as predicted by Perry's
conjecture recalled in the introduction
\cite[Conjecture 1.1]{LinRampazzoZhang}.
More generally, it is natural to ask whether positivity is needed only to recover the variety, rather than its derived category. 

\begin{conjecture}
\label{conj:derived-reconstruction}
Let $X$ and $Y$ be smooth connected complex projective varieties
such that $\OO_X$ and $\OO_Y$ are exceptional. If
\[
 \OO_X^\perp\simeq\OO_Y^\perp,
\]
then $\Db(X)\simeq\Db(Y)$.
\end{conjecture}

A stronger formulation concerns the gluing to the structure sheaf.
For the inclusion $\iota_X\colon \OO_X^\perp\hookrightarrow\Db(X)$,
let $ G_X:=\iota_X^!\OO_X\in\OO_X^\perp.$ The adjunction isomorphism
\[
 \RHom_{\OO_X^\perp}(F,G_X)\simeq\RHom_X(F,\OO_X)
\]
 where $F\in \OO_X^\perp$, shows that $G_X$ represents the gluing module. Thus the enhanced pair
$(\OO_X^\perp,G_X)$ determines $\Db(X)$ together with its distinguished
exceptional object, by dg gluing \cite{OrlovSmooth}.

\begin{conjecture}
\label{conj:gluing-reconstruction}
Let $X$ and $Y$ be as in
Conjecture~\ref{conj:derived-reconstruction}. If
$\OO_X^\perp\simeq\OO_Y^\perp$, then there exists an equivalence
\[
 \Psi\colon\OO_X^\perp\xrightarrow{\sim}\OO_Y^\perp
 \qquad\text{such that}\qquad
 \Psi(G_X)\simeq G_Y.
\]
\end{conjecture}

This would imply Conjecture~\ref{conj:derived-reconstruction} and,
more precisely, yield an equivalence sending $\OO_X$ to $\OO_Y$. For Fourier-Mukai type equivalence, Conjectures~\ref{conj:derived-reconstruction}
and~\ref{conj:gluing-reconstruction} hold for complex Enriques
surfaces, as a consequence of Li--Nuer--Stellari--Zhao
\cite[Theorem~5.1]{LiNuerStellariZhao}, applied to the decomposition
$\Db(X)=\langle\OO_X^\perp,\OO_X\rangle$.

A different question arises after removing the entire index block.
For a smooth Fano manifold with $-K_X=i_XH_X$, where $i_X$ is the
Fano index, set
\[
 \mathcal R_X:=
 \left\langle
 \OO_X,\OO_X(H_X),\ldots,
 \OO_X((i_X-1)H_X)
 \right\rangle^\perp.
\]

\begin{conjecture}
\label{conj:birational-reconstruction}
Let $X$ and $Y$ be smooth complex Fano manifolds of the same dimension. If
$\mathcal R_X\simeq\mathcal R_Y$ by a $\C$-linear equivalence
of Fourier--Mukai type, then $X$ and $Y$ are birational (or weaker, stably birational).
\end{conjecture}

For index one, Perry's conjecture predicts the stronger conclusion
$X\simeq Y$. For cubic fourfolds, the statement specializes to the
conjecture that Fourier--Mukai partners are birational; see
\cite[Conjecture 1.1]{BohningBothmerMarquand}.
The existence of nonisomorphic cubic fourfolds with equivalent
components shows why an isomorphism conclusion cannot be expected
in this generality.

\end{document}